\documentclass[10pt]{amsart}
\usepackage{bm}

\usepackage{amsmath}
\usepackage{amsthm}
\usepackage{amssymb}

\usepackage{hyperref}

\newcommand{\R}{\mathbb{R}}

\newcommand{\lr}[1]{\langle #1 \rangle}
\newcommand{\eps}{\varepsilon}

\newtheorem{thm}{Theorem}[section]
\newtheorem{prop}[thm]{Proposition}
\newtheorem{dfn}[thm]{Definition}
\newtheorem{lem}[thm]{Lemma}
\newtheorem{cor}[thm]{Corollary}

\theoremstyle{remark}
\newtheorem{rmk}[thm]{Remark}

\DeclareMathOperator{\supp}{supp}

\allowdisplaybreaks[1]

\numberwithin{equation}{section}

\date{\today}

\title[Norm inflation for Boussinesq and Kawahara eqs.]{Norm inflation for the generalized Boussinesq and Kawahara equations}

\author[M. Okamoto]{Mamoru Okamoto}
\address{Division of Mathematics and Physics, Faculty of Engineering, Shinshu University, 4-17-1 Wakasato, Nagano City 380-8553, Japan}
\email{m\_okamoto@shinshu-u.ac.jp}

\thanks{This work was supported by JSPS KAKENHI Grant number JP16K17624.}



\date{\today}

\begin{document}

\begin{abstract}
We consider ill-posedness of the Cauchy problem for the generalized Boussinesq and Kawahara equations.
We prove norm inflation with general initial data, an improvement over the ill-posedness results by Geba et al., Nonlinear Anal. 95 (2014), 404-413 for the generalized Boussinesq equations and by Kato, Adv. Differential Equations 16 (2011), no. 3-4, 257-287 for the Kawahara equation.
\end{abstract}

\maketitle

\section{Introduction}

We consider the Cauchy problem for the generalized Boussinesq equation
\begin{equation} \label{gB}
\begin{aligned}
& \partial _t^2 u - \Delta u + \Delta ^2 u + \Delta (N(u))=0,\\
& u(0,x) = u_0(x), \quad \partial _t u(0,x)=u_1(x)
\end{aligned}
\end{equation}
where $u = u(t,x) : \R \times \R^d \rightarrow \R$ is an unknown function, and $u_0$ and $u_1$ are given functions.
Falk et al. \cite{FLS} derived this equation for $d=1$ with $N(u) = 4u^3-6u^5$ in a study of shape-memory alloys.
For $N(u)=u^2$, this is the ``good'' Boussinesq equation, which arises as a model for nonlinear strings (\cite{Zak}).

In the sequel, we consider \eqref{gB} with $N(u)=u^p$.
If we ignore $\Delta u$, \eqref{gB} is invariant under the scaling transformation $u \mapsto \lambda^{\frac{2}{p-1}} u(\lambda^2t, \lambda x)$.
From
\[
\| u_{\lambda}(0,\cdot ) \|_{\dot{H}^s} = \lambda ^{s-\frac{d}{2}+\frac{2}{p-1}} \| u_0 \|_{\dot{H}^s},
\]
we call the index $s_c:= \frac{d}{2}-\frac{2}{p-1}$ scaling-critical, although the generalized Boussinesq equation does not have the exact scaling invariance.

Well-posedness of \eqref{gB} has been studied intensively for $d=1$ (see \cite{BS, Lin, Far09, Far09-2, KT, Kis13} and references therein).
Farah \cite{Far09} proved that \eqref{gB} with $d=1$ and $N(u)=u^p$ is well-posed in $H^s(\R) \times H^{s-2}(\R)$ if $d=1$, $p>1$, and $s \ge \max (s_c,0)$.
Kishimoto \cite{Kis13} showed that \eqref{gB} with $d=1$ and $N(u)=u^2$ is well-posed in $H^s(\R) \times H^{s-2}(\R)$ if $s \ge -\frac{1}{2}$.
He also proved that this result is sharp in the sense that the flow map $(u_0,u_1) \in H^s(\R) \times H^{s-2}(\R) \mapsto u(t) \in H^s(\R)$ of \eqref{gB} fails to be continuous at zero if $s<-\frac{1}{2}$.

Geba et al. \cite{GHK} proved that the flow map $(u_0,u_1) \in H^s(\R) \times H^{s-2}(\R) \mapsto u(t) \in H^s(\R)$ of \eqref{gB} fails to be $p$-times differentiable at zero if
\[
s<
\begin{cases}
-\frac{2}{p}, & \text{for $p$ odd},\\
-\frac{1}{p}, & \text{for $p$ even}.
\end{cases}
\]
It is known that the flow map is smooth if we obtain well-posedness through an iteration argument (\cite{BT}).
Hence, they showed that the standard iteration argument fails to work for \eqref{gB}.
However, as well-posedness involves the continuity of the flow map, there is a gap between ill-posedness and the presence of an irregular flow map.
In this paper, we prove ill-posedness of \eqref{gB} by observing norm inflation.

\begin{thm} \label{IP}
Let $d \in \mathbb{N}$, $p \in \mathbb{Z}_{\ge 2}$, and $N(u) = u^p$.
Assume that one of the following holds:
\begin{itemize}
\item $d=1$, $p=3$, $s \le - \frac{1}{2}$.
\item $d \in \mathbb{N}$, $p=2$, $s<-\frac{1}{2}$.
\item $d \in \mathbb{N}$, $p \ge 3$, $s<\min (s_c,0)$.
\end{itemize}
For any $(u_0, u_1) \in H^s(\R^d) \times H^{s-2}(\R ^d)$, and any $\eps >0$, there exists a solution $u_{\eps}$ to \eqref{gB} and $t_{\eps} \in (0, \eps )$ such that
\[
\| u_{\eps}(0) - u_0 \|_{H^s} + \| \partial _t u_{\eps}(0) - u_1 \|_{H^{s-2}} < \eps , \quad
\| u_{\eps}(t_{\eps}) \|_{H^s}> \eps^{-1}.
\]
In particular, the flow map of \eqref{gB} is discontinuous everywhere in $H^s(\R^d) \times H^{s-2}(\R^d)$.
\end{thm}

Theorem \ref{IP} is an improvement of the result by Geba et al. in terms of the property of the flow map and the range of $s$.

We set $v := u - i (1-\Delta)^{-1} \partial _t u$.
Since $u$ is real valued, \eqref{gB} is equivalent to
\begin{equation} \label{rSch}
\begin{aligned}
& i \partial _t v - \Delta v = -\frac{1}{2} (v-\overline{v}) + \frac{1}{2^p} \omega (\sqrt{-\Delta}) (v+\overline{v})^p , \\
& v(0,x) = v_0(x),
\end{aligned}
\end{equation}
where $\omega (\xi) = \frac{|\xi|^2}{1+|\xi|^2}$ and $v_0 = u_0- i (1-\Delta)^{-1} u_1$.
The restriction to real-valued functions is not essential, but assumed here for simplicity.
When $u$ is complex-valued, \eqref{gB} is reduced to a system of nonlinear Schr\"{o}dinger equations, and the same ill-posedness result holds (see Remark \ref{rmk:complex}).

Since $\omega (-\sqrt{\Delta})$ is bounded in $L^2(\R^d)$, we can neglect it and reduce \eqref{rSch} to the Schr\"{o}dinger equation with the power type nonlinearity.
Hence, the same calculation as the nonlinear Schr\"{o}dinger equation yields well-posedness of \eqref{gB}.
In contrast, from $\omega (\xi) \sim |\xi|$ for $|\xi| <1$, \eqref{gB} with $d=1$ and $p=2$ is well-posed in $H^{-\frac{1}{2}}(\R)$, although Kishimoto and Tsugawa \cite{KT} proved that well-posedness in $H^s(\R)$ for the nonlinear Schr\"{o}dinger equation with $|u|^2$ holds if and only if $s \ge -\frac{1}{4}$.

Iwabuchi and Ogawa \cite{IO} developed a method for proving ill-posedness of evolution equations using the modulation space.
This method is a refinement of previous work by Bejenaru and Tao \cite{BT}.
Recently, many authors have used this method to prove ill-posedness for nonlinear Schr\"{o}dinger equations \cite{Kis, IU, OW, Oh}, nonlinear Dirac equations \cite{MacOka1, MacOka2, MacOka3, HMO}, and nonlinear half wave equations \cite{CP}.

Norm inflation at general initial data was first studied by Xia \cite{Xia} in the context of the nonlinear wave equations on $\mathbb{T}^3$, which is based on the ODE approach.
Using the Fourier analytic approach, Oh \cite{Oh} proved norm inflation with general initial data for the cubic nonlinear Schr\"{o}dinger equation.
Following this approach, we show ill-posedness of \eqref{gB}.
Using well-posedness in the modulation space $M^0_{2,1}(\R^d)$ or $L^2(\R^d) \cap \mathcal{F}L^1(\R^d)$, we expand the solution $u$ to the series of the recurrence sequence $\{ \mathcal{I}_n[u_0] \}$ (see \eqref{iteration} below).
Furthermore, we can estimate the Sobolev norm of each $\mathcal{I}_n[u_0]$.
To obtain norm inflation with general initial data, Oh \cite{Oh} instead defined $\mathcal{I}_n[u_0]$ directly via the power series expansion indexed by trees.
In contrast, using the well-posedness in $M^0_{2,1}(\R^d)$ or $L^2(\R^d) \cap \mathcal{F}L^1(\R^d)$, we observe that norm inflation with general initial data follows from that with zero initial data.
As a corollary, we show that Bejenaru and Tao's argument implies that the flow map is discontinuous everywhere in a neighborhood of zero.
Since some arguments do not need to restrict the Boussinesq equation, we consider a more general setting in \S 2.

We also consider the Cauchy problem for the Kawahara equation
\begin{equation} \label{Kawahara}
\partial _t u - \partial _x^5 u+ \mathfrak{b} \partial_x^3u+\partial_x(u^2)=0, \quad u(0,x)=u_0(x),
\end{equation}
where $\mathfrak{b} =-1,0,$ or $1$.
This equation arises in modeling capillary-gravity waves on a shallow layer and magneto-sound propagation in plasmas (\cite{Kaw}).
Many authors have studied well-posedness of \eqref{Kawahara} (see \cite{CDT, WCD, CLMW, CG, Kat11} and references therein).
Kato \cite{Kat11} proved that \eqref{Kawahara} is well-posed in $H^s(\R)$ with $s \ge -2$ and ill-posed in $H^s(\R)$ with $s<-2$.
The ill-posedness means that the flow map is discontinuous at zero.
We observe discontinuity everywhere in $H^s(\R)$ with $s<-2$ for the flow map.

\begin{thm} \label{IP:Kaw}
The flow map of the Cauchy problem \eqref{Kawahara} is discontinuous everywhere in $H^s(\R)$ with $s<-2$.
More precisely, we obtain norm inflation with general initial data.
\end{thm}

We point out that Iwabuchi and Ogawa's argument is applicable even if a nonlinear component has a derivative.
Since $\omega (-\sqrt{\Delta})$ is bounded in $L^2(\R^d)$, we can reduce \eqref{rSch} to the Schr\"{o}dinger equation with a power-type nonlinearity.
Hence, well-posedness in $M^0_{2,1}(\R^d)$ or $L^2(\R^d) \cap \mathcal{F}L^1(\R^d)$ follows easily from the fact that they are a Banach algebra.
However, it is difficult to obtain well-posedness in $M^0_{2,1}(\R^d)$ or $L^2(\R^d) \cap \mathcal{F}L^1(\R^d)$ of \eqref{Kawahara} using only the product estimates in these spaces because of the presence of derivative.
To avoid discussing well-posedness, we can estimate each $\mathcal{I}_n[u_0]$ with specific initial data, which ensures the expansion.
As a result of the expansion, we obtain norm inflation with general initial data.

\begin{rmk}
For the periodic setting, through minor modifications (for example, we replace the Fourier transform by the Fourier coefficient), the proof remains valid.
Namely, under the same assumption in Theorem \ref{IP} (or Theorem \ref{IP:Kawahara}), we obtain norm inflation with general initial data in $H^s(\mathbb{T}^d) \times H^{s-2}(\mathbb{T}^d)$ (or $H^s(\mathbb{T})$) with $\mathbb{T}:= \R /\mathbb{Z}$ for \eqref{gB} (or \eqref{Kawahara}, respectively).
\end{rmk}

The remainder of this paper is organized as follows.
In \S 2, we observe the norm inflation with general initial data follows from that with zero initial data.
Furthermore, we show ill-posedness for a dispersive equation in $H^s(\R^d)$.
In \S 3, we prove Theorem \ref{IP} using the argument in \S 2.
In \S \ref{IP:Kawahara}, we prove norm inflation with general initial data for the Kawahara equation.

At the end of this section, we summarize the notations used throughout this paper.
For $1\le p \le \infty$, a Banach space $D$, and $T>0$, we denote by $L_T^p D$ the $D$-valued $L^p$ space on $[0,T]$ with the norm $\| f \|_{L_T^pD} := \| \| f \|_D \|_{L^p([0,T])}$. 
Let $\mathcal{S}(\R^d)$ be the rapidly decaying function space on $\R^d$.
We define by $\mathcal{F}[f]$ or $\widehat{f}$ the Fourier transform of $f$.
We use the inhomogeneous Sobolev spaces $H^s(\R^d)$ with the norm $\| f \|_{H^s}:= \| \lr{\cdot}^s \widehat{f} \| _{L^2}$,
where $\lr{\xi} := (1+ | \xi |^2)^{\frac{1}{2}}$.
For $1 \le p \le \infty$, we denote by $\mathcal{F}L^p (\R^d)$ the Banach space $\{ \mathcal{F} [f]: f \in L^p(\R^d) \}$ with the norm $\| f \| _{\mathcal{F}L^p} := \| \mathcal{F} f \|_{L^p}$.

We also use $\bm{1} _A$ to stand for the characteristic function of a set $A\subset \R^d$.
We use the shorthand $X \lesssim Y$ to denote the estimate $X \leq CY$ with some constant $C>0$, and $X \ll Y$ to denote the estimate $X \leq C^{-1} Y$ for some large constant $C>0$.
The notation $X \sim Y$ stands for $X \lesssim Y$ and $Y \lesssim X$.

\section{Abstract ill-posedness theory}

In this section, we deduce norm inflation with general initial data from that with zero initial data.
Since this argument does not require restriction of the Boussinesq equation, we consider the abstract (integral) equation with a $p$-linear nonlinearity for some integer $p \ge 2$:
\begin{align}
& u = \mathcal{L}(f) + \mathcal{N}_p (u, \dots , u) . \label{eq:abst1}
\end{align}
We also consider the abstract equation with a mass (and more generally bounded) component.
\begin{equation}
v = \mathcal{L}(g) + \mathcal{M}(v) + \mathcal{N}_p (v, \dots , v). \label{eq:abst2}
\end{equation}
Here, the initial data $f$, $g$ take values in some data space $D$, while the solutions $u$, $v$ take values in some solution space $S$.
The linear operators $\mathcal{L}: D \rightarrow S$ and $\mathcal{M}: S \rightarrow S$, and the $p$-linear operator $\mathcal{N}_p:S \times \dots \times S \rightarrow S$ are all densely defined.
For example, the nonlinear Schr\"{o}dinger equation $i\partial _t u- \Delta u = u^p$ is written as
\[
u (t) = e^{-i t \Delta} u_0 -i\int _0^t e^{-i (t-t') \Delta} (u^p)(t') dt'.
\]
In this case,
\[
\mathcal{L}(u_0)(t) = e^{-i t \Delta} u_0, \quad
\mathcal{N}_p(u_1, \dots , u_p) (t)= -i \int _0^t e^{-i (t-t') \Delta} (u_1 \cdots u_p)(t') dt'.
\]

To insure the local in time well-posedness in $D$ of \eqref{eq:abst1} or \eqref{eq:abst2}, we assume the following condition (A):
\begin{itemize}
\item Let $D$ and $B$ be Banach spaces with $\mathcal{S}(\R^d) \subset D \subset B \subset \mathcal{S}'(\R ^d)$.
\item There exist a suitable function space $S_T \subset C([0,T]; D)$, $\delta \in (0,1]$, and some positive constant $C_{\ast}>1$ satisfying
\begin{align*}
& \| \mathcal{L}(f) \| _{S_T} \le C_{\ast} \| f \| _{D}, \quad \| \mathcal{L}(f) \|_{L_T^{\infty} B} \le C_{\ast} \| f \| _B, \\
& \| \mathcal{M}(u) \| _{S_T} \le C_{\ast} T^{\delta} \| u \|_{S_T}, \quad \| \mathcal{M}(u) \| _{L_T^{\infty} B} \le C_{\ast} T^{\delta} \| u \|_{B}, \\
& \left\| \mathcal{N}_p (u_1, \dots , u_p) \right\| _{S_T} \le C_{\ast} T^{\delta} \prod _{k=1}^p \| u_k \| _{S_T}
\end{align*}
for $0<T<1$.
\end{itemize}
Under this assumption, well-posedness of \eqref{eq:abst1} or \eqref{eq:abst2} follows from the standard contraction mapping argument.
For example, we can find a solution $v$ in $\{ v \in S_T: \| v \| _{S_T} \le 2 C_{\ast} \| g \|_D \}$ to \eqref{eq:abst2} provided that
\[
T < \left\{ p2^pC_{\ast}^p (\| g \| _D^{p-1}+1) \right\} ^{-\frac{1}{\delta}} .
\]

\begin{dfn}
We say that norm inflation of \eqref{eq:abst2} with general initial data occurs if for any $\eps >0$ and any $\phi \in B$, there exist a solution $v_{\eps}$ to \eqref{eq:abst2} with initial data $v_{\eps}(0) \in \mathcal{S} (\R^d)$ and a time $t_{\eps} \in \R$ such that
\[
\| v_{\eps}(0) - \phi \| _B < \eps, \quad 0<t_{\eps} <\eps , \quad \| v_{\eps}(t_{\eps}) \|_B > \eps ^{-1}.
\]
\end{dfn}
We note that the norm inflation of \eqref{eq:abst2} with general initial data follows from the norm inflation of \eqref{eq:abst1} at zero.

\begin{prop} \label{prop:gIP}
Assume (A) and that for any $k \in \mathbb{N}$, there exists a solution $u_k$ to \eqref{eq:abst1} with initial data $u_k(0) \in \mathcal{S}(\R^d)$ and $t_k \in \R$ such that
\[
10^p C_{\ast}^p t_k^{\delta} \| u_k(0) \| _D^{p-1}<1, \quad
\| u_k (0) \| _B < \frac{1}{k}, \quad
\| u_k (t_k) \| _B >k.
\]
Norm inflation of \eqref{eq:abst2} with general initial data occurs.
\end{prop}

\begin{proof}
For any $\varphi \in \mathcal{S}(\R^d)$, we take $v_k(0) = u_k(0) + \varphi \in \mathcal{S}(\R^d)$.
For simplicity, we set $u_{0,k}:= u_k(0)$ and $v_{0,k} := u_k(0)+\varphi$.
By the well-posedness of \eqref{eq:abst1} in $D$, we have
\[
k< \| u_k(t_k)\| _B \le \| u_k \| _{L^{\infty}_{t_k}D} \le 2 C_{\ast} \| u_{0,k} \|_D.
\]
Since $\varphi$ is independent of $k$, we have $\| v_{0,k} \|_D \le 2 \| u_{0,k} \|_D$ for sufficiently large $k$.
Therefore, by the assumption, \eqref{eq:abst2} is well-posed in $C([0,t_k];D)$.
Furthermore, we get
\[
\| u_k \| _{S_{t_k}} \le 2C_{\ast} \| u_{0,k} \| _D, \quad
\| v_k \| _{S_{t_k}} \le 2C_{\ast} \| u_{0,k} \| _D, \quad
t_k < (5^p C_{\ast} k^{p-1})^{-\frac{1}{\delta}} .
\]

The difference $v_k-u_k$ satisfies
\[
v_k-u_k = \mathcal{L}(\varphi) + \mathcal{M}(v_k) + \mathcal{N}_p(v_k,\dots ,v_k)-\mathcal{N}_p(u_k,\dots ,u_k) .
\]
Therefore,
\begin{align*}
& \| v_k-u_k \| _{L_{t_k}^{\infty}B} \\
& \le \| \mathcal{L}(\varphi ) \| _{S_{t_k}} + \| \mathcal{M}(v_k) \| _{L_{t_k}^{\infty}B} + \| \mathcal{N}_p(v_k,\dots ,v_k)-\mathcal{N}_p(u_k,\dots ,u_k) \| _{S_{t_k}} \\
& \le C_{\ast} \| \varphi \|_D + C_{\ast} t_k^{\delta} \| v_k \|_{L_{t_k}^{\infty}B} + p C_{\ast} t_k^{\delta} (\| v_k \|_{S_{t_k}} + \| u_k \|_{S_{t_k}})^{p-1} \| v_k - u_k \|_{S_{t_k}} \\
& \le C_{\ast} \| \varphi \|_D + \frac{1}{2} \| v_k \|_{L_{t_k}^{\infty}B} + p C_{\ast} t_k^{\delta} (4C_{\ast} \| u_{0,k} \|_D) ^{p-1} \| v_k - u_k \|_{S_{t_k}} \\
& \le C_{\ast} \| \varphi \|_D + \frac{1}{2} \| v_k \|_{L_{t_k}^{\infty}B} + \frac{1}{2} \| v_k - u_k \|_{S_{t_k}},
\end{align*}
and thus,
\[
\| v_k - u_k \|_{L_{t_k}^{\infty}B} \le 2C_{\ast} \| \varphi \|_D+ \| v_k \|_{L_{t_k}^{\infty}B} .
\]
From the triangle inequality
\[
\| v_k-u_k \|_{L_{t_k}^{\infty}B}
\ge \| u_k \| _{L_{t_k}^{\infty}B} - \| v_k \| _{L_{t_k}^{\infty}B},
\]
it follows that
\[
\| v_k \| _{L_{t_k}^{\infty}B} 
\ge \frac{1}{2} \| u_k \| _{L_{t_k}^{\infty}B} -C_{\ast} \| \varphi \|_D
> \frac{k}{4} .
\]
for sufficiently large $k$.
Thus, we can find $t_k' \in (0, t_k]$ satisfying
\[
\| v_k (t_k') \| _B > \frac{k}{8} .
\]

For completeness, we deduce norm inflation of \eqref{eq:abst2} with general initial data from the above argument.
Let $\phi \in B$.
For any $k \in \mathbb{N}$, there exists $\varphi_k \in \mathcal{S}(\R^d)$ such that
\[
\| \phi - \varphi_k \|_B < \frac{1}{k}.
\]
By the above argument, there exists a solution $v_k$ to \eqref{eq:abst2} such that
\[
\| v_k(0) - \varphi_k \|_B < \frac{1}{k}, \quad \| v_k (t_k') \| _B > \frac{k}{8}
\]
for sufficiently large $k$.
Hence, for given $\eps >0$, taking $v_{\eps} = v_{k}$ and $t_{\eps} = t_k'$ for some $k > \max \{ 8\eps^{-1}, (5^pC_{\ast} \eps^{\delta})^{-\frac{1}{p-1}} \}$, we obtain
\[
\| v_{\eps}(0) - \phi \| _B \le \frac{2}{k} < \eps , \quad
0< t_{\eps} < (5^p C_{\ast} k^{p-1})^{-\frac{1}{\delta}} < \eps , \quad
\| v_{\eps} (t_{\eps}) \| _B > \frac{k}{8} > \eps^{-1} ,
\]
concluding the proof.
\end{proof}

As a corollary, we get that the flow map is discontinuous everywhere in a neighborhood of zero if that is discontinuous at zero.

\begin{cor} \label{cor:gIP}
Assume (A) and that for any $k \in \mathbb{N}$, there exist $C>0$ and a solution $u_k$ to \eqref{eq:abst1} with initial data $u_k(0) \in \mathcal{S}(\R^d)$  such that
\[
10^p C_{\ast}^p \| u_k(0) \| _D^{p-1}<1, \quad
\| u_k (0) \| _B < \frac{1}{k}, \quad
\| u_k \| _{L_T^{\infty}B} > C
\]
for $0<T\ll 1$.
Then, the flow map of \eqref{eq:abst2} is discontinuous everywhere in $\{ u_0 \in B : \| u_0 \|_B < \frac{C}{2 C_{\ast}} \}$.
\end{cor}

The assumption of this corollary is satisfied if ill-posedness is proved by Bejenaru and Tao's argument.
For example, by \cite{Kat11} and this corollary, we obtain that the flow map of the Kawahara equation \eqref{Kawahara} is discontinuous everywhere in a neighborhood of zero in $H^s(\R)$ with $s<-2$.
Here, we take $S_T$ and $D$ as the modified Fourier restriction norm space defined in \cite{Kat11} and $H^{-2}(\R)$, respectively.
In fact, we also obtain norm inflation with general initial data for the Kawahara equation (see \S \ref{IP:Kawahara} below).

We present an application of Proposition \ref{prop:gIP}.
Note that assumption (A) implies that the solution $u$ to \eqref{eq:abst1} is expanded as follows (see \cite{BT}).
\[
u = \sum _{n=1}^{\infty} \mathcal{I}_n[f] \quad \text{in } S_T
\]
where
\begin{equation} \label{iteration}
\mathcal{I}_1[f] := \mathcal{L}(f), \quad
\mathcal{I}_n[f] := \sum _{\substack{n_1, \dots , n_p \in \mathbb{N} \\ n_1+\dots +n_p=n}} \mathcal{N}_p(\mathcal{I}_{n_1}[f], \dots , \mathcal{I}_{n_p}[f]) \quad \text{for }  n \in \mathbb{Z}_{\ge 2} .
\end{equation}
Here, $\mathcal{I}_n[f]$ becomes zero unless $n=l (p-1)+1$ for $l \in \mathbb{N} \cup \{ 0 \}$.
To show that well-posedness by the iteration argument breaks down, we usually prove failure of the bound
\[
\| \mathcal{I}_p [f] \| _{L_T^{\infty} B} \le C \| f \| _B^p .
\]
This implies that the flow map of \eqref{eq:abst1} fails to be $p$-times differentiable.
Thus, if $\mathcal{I}_p[f]$ is a leading term in the expansion, then
\[
\| u \| _{L_T^{\infty} B} \ge \| \mathcal{I}_p [f] \| _{L_T^{\infty} B} - \| \mathcal{I}_1[f] \|_{L_T^{\infty} B} - \sum _{n=2p-1}^{\infty} \| \mathcal{I}_n [f] \| _{L_T^{\infty} B}
\sim \| \mathcal{I}_p [f] \| _{L_T^{\infty} B} ,
\]
implying the discontinuity of the flow map.
This intuitive argument is valid in some cases, as shown below in a specific setting.

To demonstrate this phenomenon, we specify the problem.
Let $\varphi \in C^{\infty} (0,\infty)$ be a real valued function satisfying $\varphi' >0$ and $|\varphi (r)| \sim r^{\alpha}$ for any $r>1$ and some $\alpha >0$.
We abbreviate $\varphi (|\xi|)$ to $\varphi (\xi)$ for $\xi \in \R^d$.
We consider the dispersive equation, namely,
\[
\mathcal{L}(u_0) = e^{i t \varphi (-i \nabla )} u_0 , \quad
\mathcal{N}_p(u_1, \dots ,u_p) = \int _0^t e^{i (t-t') \varphi (-\nabla)} N_p(u_1, \dots , u_p) (t') dt' ,
\]
where $N_p$ is a monomial with respect to $u_1$, $\overline{u}_1$, \dots , $u_p$, $\overline{u}_p$.
Taking $D=L^2(\R^d) \cap \mathcal{F} L^1(\R^d) $ and $S_T = L_T^{\infty} D$ as in \cite{Oh}, we have
\[
\| \mathcal{L}(u_0) \|_{S_T} = \| u_0 \|_D, \quad
\| \mathcal{N}_p(u_1, \dots ,u_p) \| _{S_T} \le T \prod _{j=1}^p \| u_j \| _{S_T}.
\]
Therefore, the condition (A) holds with $\delta =1$ and $B=H^s(\R^d)$ if $s<0$.
We can also take $D$ as the modulation space $M^0_{2,1}(\R^d)$ as in \cite{IO, Kis}.

\begin{rmk} \label{rmk:persistence}
Since $\mathcal{F}L^1 (\R^d)$ is a Banach algebra and $D \subset \mathcal{F} L^1(\R^d)$, from the persistence, the solution $u$ in $\{ u \in S_T: \| u \| _{S_T} \le 2 C_{\ast} \| u_0 \|_D \}$ to \eqref{eq:abst2} exists provided that
\[
T < \left\{ p2^p(\| u_0 \| _{\mathcal{F}L^1}^{p-1}+1) \right\} ^{-1} .
\]
Accordingly, we can replace $10^p C_{\ast}^pt_k^{\delta} \| u_k(0) \| _D^{p-1}<1$ in Proposition \ref{prop:gIP} by
\[
10^p t_k (\| u_k(0) \| _{\mathcal{F}L^1}^{p-1}+k)<1.
\]
The condition $10^p t_k k<1$ is required because we are not certain whether $\| u_k(0) \| _{\mathcal{F}L^1} \ge 1$ holds or not.
\end{rmk}

Since $N_p $ is a monomial with respect to $u_1$, $\overline{u}_1$, \dots , $u_p$, $\overline{u}_p$, we can write
\[
N_p(u, \dots ,u) \sim u^m \overline{u}^{p-m}
\]
for some integer $m \in \{ 0, 1, \dots , p\}$.
We thus focus on the Cauchy problem
\begin{equation} \label{eq:dispersive}
\partial _t u - i \varphi (-\nabla) u = \mathfrak{a} u^{p-m}\overline{u}^m, \quad u(0) =u_0 ,
\end{equation}
where $\mathfrak{a} \in \mathbb{C}$ with $|\mathfrak{a}|=1$.

\begin{prop} \label{prop:disIP}
Let $d \in \mathbb{N}$, $\alpha >0$, $p \in \mathbb{Z}_{\ge 2}$, and $m \in \{ 0,1 \dots , p\}$.
Then, the flow map of \eqref{eq:dispersive} is discontinuous everywhere in $H^s(\R^d)$ if one of the following conditions holds:
\begin{itemize}
\item $p>1+\frac{\alpha}{d}$, $s<0$, and $s<\frac{d}{2}-\frac{\alpha}{p-1}$.
\item $p=1+\frac{\alpha}{d}$ and $s \le -\frac{d}{2}$.
\item $p < 1+\frac{\alpha}{d}$ and $s<\frac{d}{2}-\frac{d+\alpha}{p}$.
\end{itemize}
More precisely, norm inflation with general initial data occurs.
\end{prop}

If $\varphi (r) = r^{\alpha}$, then \eqref{eq:dispersive} is invariant under the scale transformation
\[
u \mapsto u^{(\lambda)}(t,x) := \lambda ^{\frac{\alpha}{p-1}} u(\lambda ^{\alpha}t , \lambda x),
\]
implying that the scaling-critical Sobolev index is $\frac{d}{2}-\frac{\alpha}{p-1}$.
The index $\frac{d}{2}-\frac{d+\alpha}{p} = d (\frac{1}{2}-\frac{1}{p})-\frac{\alpha}{p}$ appears for a technical reason.
Roughly speaking, this is a reflection of the modulation estimate (see \eqref{est:mod} below) and the embedding $H^{d (\frac{1}{2}-\frac{1}{p})}(\R^d) \hookrightarrow L^p(\R^d)$, which ensures well-definedness of the nonlinearity in the distribution sense.
We note that latter is bigger than the former provided that $p<1+\frac{\alpha}{d}$.
Furthermore, they are equal if $p=1+\frac{\alpha}{d}$, which leads to ill-posedness in $H^{-\frac{d}{2}}(\R^d)$.

Before proving this proposition, we apply it to ill-posedness for the quadratic nonlinear Schr\"{o}dinger equation
\[
i \partial _t u -  \Delta u = u^2 \quad \text{or} \quad i \partial _t u - \Delta u = \overline{u}^2.
\]
We obtain norm inflation with general initial data in $H^s(\R^d)$ if one of the following holds:
\begin{itemize}
\item $d=1$ and $s<-1$.
\item $d=2$ and $s\le -1$.
\item $d \ge 3$, $s<0$, and $s<\frac{d}{2}-2$. 
\end{itemize}
For norm inflation with zero initial data, see \cite{BT, IO}.

We estimate each $\mathcal{I}_n[u_0]$, as shown in the modulation spaces by Kishimoto \cite{Kis} (see also \cite{Oh}).

\begin{lem} \label{lem:bound}
There exists $C_1>1$ depending only on $d$ and $p$ such that
\begin{align*}
& \| \mathcal{I}_n [u_0] (t) \| _{\mathcal{F}L^1} \le (C_1 t^{\frac{1}{p-1}} \| u_0 \| _{\mathcal{F}L^1} )^{n-1} \| u_0 \|_{\mathcal{F}L^1} , \\
& \| \mathcal{I}_n [u_0] (t) \| _{L^2} \le (C_1 t^{\frac{1}{p-1}} \| u_0 \|_{\mathcal{F}L^1} ) ^{n-1} \| u_0 \|_{L^2}
\end{align*}
for $t>0$ and $u_0 \in L^2(\R^d) \cap \mathcal{F}L^1(\R^d)$.
\end{lem}

The proof is essentially reduced to the following bound.

\begin{lem}[Kishimoto \cite{Kis}] \label{lem:bound-seq}
Let $\{ a_n \}$ be a sequence of nonnegative real numbers such that
\[
a_n \le C \sum _{\substack{n_1, \dots , n_p \in \mathbb{N} \\ n_1+ \dots +n_p=n}} a_{n_1} \dots a_{n_p}
\]
for some $p \in \mathbb{Z}_{\ge 2}$ and $C>0$.
Then,
\[
a_n \le a_1 C_0^{n-1}
\]
with
\[
C_0 := \frac{\pi^2}{6} (Cp^2)^{\frac{1}{p-1}}a_1.
\]
\end{lem}

\begin{proof}[Proof of Lemma \ref{lem:bound}]
Let $\{ a_n \}$ be the sequence defined by
\[
a_1 = 1, \quad
a_n = \frac{p-1}{n-1} \sum _{\substack{n_1, \dots , n_p \in \mathbb{N} \\ n_1+ \dots +n_p=n}} a_{n_1} \dots a_{n_p}
\]
for $n\ge 2$.
Owing to Lemma \ref{lem:bound-seq}, it suffices to show that
\begin{align}
& \| \mathcal{I}_n [u_0] (t) \| _{\mathcal{F}L^1} \le a_n ( t^{\frac{1}{p-1}} \| u_0 \| _{\mathcal{F}L^1} ) ^{n-1} \| u_0 \|_{\mathcal{F}L^1} , \notag \\
& \| \mathcal{I}_n [u_0] (t) \| _{L^2} \le a_n (t^{\frac{1}{p-1}} \| u_0 \| _{\mathcal{F}L^1} ) ^{n-1} \| u_0 \|_{L^2} . \label{bound:L2}
\end{align}
We use induction.
Since $n=1$ is trivial, we assume that this estimate holds up to $n-1$.
Then, from the induction hypothesis, we have
\begin{align*}
\| \mathcal{I}_n [u_0] (t) \| _{\mathcal{F}L^1}
& \le \sum _{\substack{n_1, \dots , n_p \in \mathbb{N} \\ n_1+\dots +n_p=n}} \int_0^t \bigg\| N_p (\mathcal{I}_{n_1}[u_0] (t'), \dots , \mathcal{I}_{n_p}[u_0] (t')) \bigg\| _{\mathcal{F}L^1} dt' \\
& \le \sum _{\substack{n_1, \dots , n_p \in \mathbb{N} \\ n_1+\dots +n_p=n}} \int_0^t \prod_{j=1}^p \bigg\| \mathcal{I}_{n_j}[u_0] (t') \bigg\| _{\mathcal{F}L^1} dt' \\
& \le \sum _{\substack{n_1, \dots , n_p \in \mathbb{N} \\ n_1+\dots +n_p=n}} \int_0^t \prod _{j=1}^p \bigg\{ a_j ( t'^{\frac{1}{p-1}} \| u_0 \| _{\mathcal{F}L^1} ) ^{n_j-1} \| u_0 \|_{\mathcal{F}L^1} \bigg\} dt' \\
& = a_n ( t^{\frac{1}{p-1}} \| u_0 \|_{\mathcal{F}L^1} ) ^{n-1} \| u_0 \|_{\mathcal{F}L^1} .
\end{align*}
The estimate \eqref{bound:L2} is obtained in the same manner.
\end{proof}

To show estimates in the Sobolev spaces, we specify an initial datum.
We denote by $Q_r(q)$ the cube with length $2r$ centered at $(q, 0, \dots , 0)$, i.e.,
\[
Q_r(q) := [q-r,q+r] \times \prod _{j=2}^d [-r,r].
\]
Put $\widetilde{Q}_r(q) := Q_r(q) \cup Q_r(-q)$.
For positive numbers $N, \, A$ with $N \ge \max (2A, 2)$, we take the initial datum
\[
u_0 = (\log N)^{-\frac{1}{16}} N^{-s} A^{-\frac{d}{2}} \mathcal{F}^{-1} [ \bm{1}_{\widetilde{Q}_A(N)} + \bm{1}_{\widetilde{Q}_A(2N)} ] .
\]
Then,
\[
\| u_0 \| _{H^s} \sim (\log N)^{-\frac{1}{16}}, \quad
\| u_0 \| _{\mathcal{F}L^1} \sim (\log N)^{-\frac{1}{16}} N^{-s} A^{\frac{d}{2}} .
\]
Here, we recall the following simple estimate.

\begin{lem} \label{lem:conv}
For any $q_1, \, q_2 \in \R^d$, and $r>0$, there exist constants $c_1>0$ and $c_2>0$ depending only on $d$ such that
\[
c_1 r^d \bm{1}_{Q_r(q_1+q_2)} \le \bm{1}_{Q_r(q_1)} \ast \bm{1}_{Q_r(q_2)} \le c_2 r^d \bm{1}_{Q_{2r}(q_1+q_2)} .
\]
\end{lem}

Lemma \ref{lem:bound-seq} yields estimates in the Sobolev space.
We remark that similar estimates are proved in \cite{Kis, Oh}.

\begin{lem} \label{lem:boundHs}
For any $s<0$, there exists $C$ depending only on $d$, $p$, $s$ such that for $n \ge 2$, we have
\[
\| \mathcal{I}_n[u_0] (t) \| _{H^s} \le C^n t^{\frac{n-1}{p-1}} \| u_0 \|_{\mathcal{F}L^1}^{n-2} \| u_0 \|_{L^2}^2 F_s(A) ,
\]
where
\[
F_s(A) := \begin{cases} A^{s+\frac{d}{2}}, & \text{if } -\frac{d}{2}<s<0, \, A \ge 1, \\ (\log \lr{A})^{\frac{1}{2}}, & \text{if } s=-\frac{1}{2}, \, A \ge 1, \\ 1, & \text{if } s<-\frac{d}{2}, \, A \ge 1, \\  A^{\frac{d}{2}} , & \text{if } A \le 1. \end{cases}
\]
\end{lem}

\begin{proof}
Lemma \ref{lem:conv} implies that the support of $\mathcal{F}[\mathcal{I}_n[u_0]]$ consists of at most $4^{n}$ cubes of volume $(2nA)^d$.
Hence, there exists $C_2>0$ depending only on $d$ such that
\[
| \supp \mathcal{F}[\mathcal{I}_n[u_0]] |\le 4^n (2nA)^d = n^d 2^{2n+d} A^d < C_2^n A^d.
\]
From $s<0$ and $A \le \frac{N}{2}$, we get
\[
\| \lr{\cdot}^s\| _{L^2_{\xi}(\supp \mathcal{F}[\mathcal{I}_n[u_0]])}
\le 4^n \| \lr{\cdot}^s\| _{L^2_{\xi}(Q_{nA}(0))}
\le C_3^n F_s(A),
\]
where $C_3=C_3(d,s)>0$.

By Young' inequality and Lemma \ref{lem:bound},
\begin{align*}
& \| \mathcal{I}_n[f] (t) \| _{\mathcal{F}L^{\infty}} \\
& \le \sum _{\substack{n_1, \dots , n_p \in \mathbb{N} \\ n_1+\dots +n_p=n}} \int _0^t \bigg\| N_p (\mathcal{I}_{n_1}[u_0] (t'), \dots , \mathcal{I}_{n_p}[u_0] (t')) \bigg\| _{\mathcal{F}L^{\infty}} dt' \\
& \le \sum _{\substack{n_1, \dots , n_p \in \mathbb{N} \\ n_1+\dots +n_p=n}} \int _0^t \| \mathcal{I}_{n_1}[f] (t')  \| _{L^2} \| \mathcal{I}_{n_2}[f] (t')  \| _{L^2} \prod _{j=3}^p \| \mathcal{I}_{n_j}[f] (t') \| _{\mathcal{F} L^1} dt' \\
& \le \sum _{\substack{n_1, \dots , n_p \in \mathbb{N} \\ n_1+\dots +n_p=n}} \int _0^t (C_1t'^{\frac{1}{p-1}} \| u_0\| _{\mathcal{F}L^1})^{n_1-1} \| u_0 \| _{L^2} (C_1t'^{\frac{1}{p-1}} \| u_0\|_{\mathcal{F}L^1})^{n_2-1} \| u_0 \| _{L^2} \\
& \hspace*{100pt} \times \prod _{j=3}^p (C_1t'^{\frac{1}{p-1}} \| u_0 \|_{\mathcal{F}L^1})^{n_j-1} \| u_0 \|_{\mathcal{F}L^1} dt' \\
& \le n^p C_1^{n-p} t^{\frac{n-1}{p-1}} \| u_0 \|_{\mathcal{F}L^1}^{n-2} \| u_0 \| _{L^2}^2 .
\end{align*}
Combining the estimates above with
\[
\| \mathcal{I}_n[u_0] \| _{H^s}
\le \| \lr{\cdot}^s\| _{L^2_{\xi}(\supp \mathcal{F}[\mathcal{I}_n[u_0]])} \| \mathcal{I}_n [u_0] \|_{\mathcal{F} L^{\infty}} ,
\]
we obtain the desired bound.
\end{proof}

\begin{proof}[Proof of Proposition \ref{prop:disIP}]
A direct calculation yields
\begin{align*}
& \mathcal{F}[\mathcal{I}_p[u_0]](t,\xi) \\
& = \mathfrak{a} \int _0^t e^{i (t-t') \varphi (\xi)} \int_{\xi_1+\dots + \xi_p=\xi} \left( \prod _{j=1}^m e^{it' \varphi (\xi_j)} \widehat{u}_0(\xi_j) \right) \left( \prod _{l=m+1}^p \overline{e^{it' \varphi (\xi_l)} \widehat{u}_0(\xi_l)} \right) dt' \\
& = \mathfrak{a} e^{it \varphi (\xi)} \int_{\ast} \int _0^t e^{ -it' M(\xi_1, \dots , \xi_p)} dt' \prod _{j=1}^p \widehat{u}_0(\xi_j) ,
\end{align*}
where
\[
\int _{\ast} = \int _{\xi_1+\dots + \xi_p=\xi}, \quad
M(\xi_1,\dots , \xi_p) = \varphi \left( \sum _{j=1}^p \xi_l \right) - \sum _{j=1}^m \varphi (\xi_j) + \sum _{l=m+1}^p \varphi (\xi_l) .
\]
Since $\xi _j \in \widetilde{Q}_A(N) \cup \widetilde{Q}_A(2N)$, we have
\begin{equation} \label{est:mod}
|M(\xi_1, \dots , \xi_p)|
\lesssim |\varphi (N)| \sim N^{\alpha}
\end{equation}
for $|\xi | \le A$.
Hence, taking $0<t \ll N^{-\alpha}$, we have
\[
\Re \int _0^t e^{ -it' M(\xi_1, \dots , \xi_p)} dt'  > \frac{t}{2}.
\]
Therefore, by Lemma \ref{lem:conv},
\begin{align*}
\| \mathcal{I}_p[u_0] (t) \| _{H^s}
& \ge \| \lr{\cdot}^s \mathcal{F}[\mathcal{I}_p[u_0]](t) \| _{L^2(Q_A(0))} \\
& \gtrsim t (\log N)^{-\frac{p}{16}} (N^{-s} A^{-\frac{d}{2}})^p A^{(p-1)d} F_s(A) \\
& \sim t (\log N)^{-\frac{p}{16}} (N^{-s} A^{\frac{d}{2}})^{p-2} N^{-2s} F_s(A) \\
& \sim t \| u_0 \|_{\mathcal{F}L^1}^{p-2} \| u_0 \|_{L^2}^2 F_s(A) .
\end{align*}
Hence, it suffices to choose $t=t(N)>0$ and $A=A(N) \in [1, N]$ satisfying
\[
t \| u_0 \|_{\mathcal{F}L^1}^{p-1} \ll 1, \quad
t \ll N^{-\alpha} , \quad
t \| u_0 \|_{\mathcal{F}L^1}^{p-2} \| u_0 \|_{L^2}^2 F_s(A) \gtrsim (\log N)^{\frac{1}{8}} .
\]
Indeed, by the triangle inequality and Lemma \ref{lem:boundHs}, we obtain
\begin{align*}
\| u(t) \| _{H^s}
& \ge \| \mathcal{I}_p[u_0](t) \|_{H^s} - \| \mathcal{I}_1[u_0] (t) \|_{H^s} - \sum _{n=2p-1}^{\infty} \| \mathcal{I}_n[u_0] (t) \|_{H^s} \\
& \gtrsim t \| u_0 \|_{\mathcal{F}L^1}^{p-2} \| u_0 \|_{L^2}^2 F_s(A) - \| u_0 \|_{H^s} - t \| u_0 \|_{\mathcal{F}L^1}^{p-2} \| u_0 \|_{L^2}^2 F_s(A) \times t \| u_0 \|_{\mathcal{F}L^1}^{p-1} \\
& \gtrsim (\log N)^{\frac{1}{8}}.
\end{align*}
From Proposition \ref{prop:gIP} and Remark \ref{rmk:persistence}, we obtain norm inflation with general initial data.

It remains to verify the above three conditions.
We take
\[
t \sim (\log N)^{-\frac{1}{8}} \min ( \| u_0 \|_D^{-(p-1)}, N^{-\alpha}) .
\]
Then, the first and second conditions are fulfilled.
The third inequality is reduced to
\begin{equation} \label{cond:3}
\min \left\{ (\log N)^{-\frac{3}{16}} N^{-s} A^{-\frac{d}{2}} , (\log N)^{-\frac{p+2}{16}} N^{-ps-\alpha} A^{\frac{d(p-2)}{2}} \right\} F_s(A) \gtrsim (\log N)^{\frac{1}{8}}.
\end{equation}
Here,
\begin{align*}
& N^{-s} A^{-\frac{d}{2}} F_s(A) = \begin{cases} \left( \frac{N}{A} \right)^{-s} , & \text{if } -\frac{d}{2}<s<0, \\ \left( \frac{N}{A} \right)^{\frac{d}{2}} (\log A)^{\frac{1}{2}} , & \text{if } s=-\frac{d}{2}, \\ N^{-s} A^{-\frac{d}{2}}, & \text{if } s<-\frac{d}{2}, \end{cases} \\
& N^{-ps-\alpha} A^{\frac{d(p-2)}{2}} F_s(A) = \begin{cases} N^{-ps-\alpha} A^{s+\frac{d(p-1)}{2}}, & \text{if } -\frac{d}{2}<s<0, \\ N^{\frac{dp}{2}-\alpha} A^{\frac{d(p-2)}{2}} (\log A)^{\frac{1}{2}}, & \text{if } s=-\frac{d}{2}, \\ N^{-ps-\alpha} A^{\frac{d(p-2)}{2}}, & \text{if } s<-\frac{d}{2}. \end{cases}
\end{align*}
From the assumption, we can take a real number $\theta$ satisfying
\[
\begin{cases} \max \left( \frac{ 2(ps+\alpha )}{2s+d (p-1)} ,0 \right) < \theta <1, & \text{if } -\frac{d}{2}< s <0, \\
\max \left( \frac{2(ps+\alpha)}{d(p-2)} ,0 \right) < \theta < 1, & \text{if } s \le -\frac{d}{2}, \, p>2, \\
0, & \text{if } s < -\frac{d}{2}, \, p=2.
\end{cases}
\]
We note that $s=-\frac{d}{2}$ and $p=2$ imply $d \ge \alpha$.
In fact, if $d < \alpha$, then $p=2<1+\frac{\alpha}{d}$.
Hence, the assumption is reduced to $s<-\frac{\alpha}{2}$, which contradicts to $s=-\frac{d}{2}$ and $d< \alpha$.
Therefore, setting
\[
A= \begin{cases} (\log N)^{-1} N, & \text{if } s=-\frac{d}{2}, \, p=2, \\ N^{\theta}, & \text{otherwise}, \end{cases}
\]
by the assumption, we obtain \eqref{cond:3} for sufficiently large $N \in 2^{\mathbb{N}}$.
\end{proof}

The condition $\varphi (r) \sim r^{\alpha}$ for $r>1$ is used only in the lower bound of $\| \mathcal{I}_p [u_0] \|_{H^s}$, more precisely the modulation bound \eqref{est:mod}.
Hence, we obtain ill-posedness in $H^s(\R^d)$ with $s$ larger if \eqref{est:mod} is improved.

Let
\[
Q_r^{(1)}(q) := [q-r,q+r] \times \prod _{j=2}^d [-1,1]
\]
for $r>0$ and $q \in \R$.
Put $\widetilde{Q}_r^{(1)}(q) := Q_r^{(1)}(q) \cup Q_r^{(1)}(-q)$.
If there exists $\beta \in (0, \alpha]$ such that
\begin{equation} \label{modulation}
|M(\xi_1,\dots , \xi_p)|  \lesssim N^{\alpha - \beta} A^{\beta} +1
\end{equation}
for any $\xi _j \in \widetilde{Q}_A^{(1)}(N)$ with $\xi_1+\dots \xi_p \in Q_A^{(1)}(0)$, then we obtain the following.

\begin{cor} \label{cor:disIP}
In addition to the assumption in Proposition \ref{prop:disIP}, we assume that \eqref{modulation} holds.
Then, the flow map of the Cauchy problem \eqref{eq:dispersive} is discontinuous everywhere in $H^s(\R^d)$ if $s < \frac{\alpha - \beta}{2\beta} \left( \frac{1}{p} -1 \right)$.
More precisely, norm inflation with general initial data occurs.
\end{cor}

\begin{proof}
Let
\[
u_0 = (\log N)^{-\frac{1}{16}} N^{-s} A^{-\frac{1}{2}} \mathcal{F}^{-1} [ \bm{1}_{\widetilde{Q}_A^{(1)}(N)}]
\]
for $A \le 1 < N$.
Then,
\[
\| u_0 \| _{H^s} \sim (\log N)^{-\frac{1}{16}}, \quad
\| u_0 \| _{\mathcal{F}L^1} \sim (\log N)^{-\frac{1}{16}} N^{-s} A^{\frac{1}{2}} .
\]
As in Lemma \ref{lem:boundHs}, for any $s<0$, there exists $C$ depending only on $d$, $p$, $s$ such that for $n \ge 2$, we have
\[
\| \mathcal{I}_n[u_0] (t) \| _{H^s} \le C^n t^{\frac{n-1}{p-1}} \| u_0 \|_{\mathcal{F}L^1}^{n-2} \| u_0 \|_{L^2}^2 A^{\frac{1}{2}}.
\]
Since
\[
\left | \int _0^t e^{ -it' M(\xi_1, \dots , \xi_p)} dt' \right|
\gtrsim t
\]
for $0< t \ll \lr{N^{\alpha-\beta} A^{\beta}}^{-1}$, as in Proposision \ref{prop:disIP}, it suffices to choose $t=t(N)>0$ and $A=A(N) \le 1$ satisfying
\[
t \| u_0 \|_{\mathcal{F}L^1}^{p-1} \ll 1 , \quad
t \lr{N^{\alpha - \beta} A^{\beta}} \ll 1 , \quad
t \| u_0 \|_{\mathcal{F}L^1}^{p-2} \| u_0 \|_{L^2}^2 A^{\frac{1}{2}} \gtrsim (\log N)^{\frac{1}{8}} .
\]
We take
\[
t \sim (\log N)^{-\frac{1}{8}} \min ( \| u_0 \|_{\mathcal{F}L^1}^{-(p-1)}, \lr{N^{\alpha - \beta} A^{\beta}}^{-1}) ,\quad
A = N^{-\frac{\alpha - \beta}{\beta}}.
\]
Then, the first and second conditions are fulfilled.
The third inequality is reduced to
\[
\min \left\{ (\log N)^{-\frac{3}{16}} N^{-s} , (\log N)^{-\frac{p+2}{16}} N^{-ps- \frac{\alpha-\beta}{2\beta} (p-1)} \right\} \gtrsim (\log N)^{\frac{1}{8}} .
\]
This is fulfilled provided that $s<\frac{\alpha-\beta}{2\beta}(\frac{1}{p}-1)$.
\end{proof}

For example, the flow map of the Cauchy problem for the quadratic nonlinear Schr\"{o}dinger equation
\[
i \partial _tu - \Delta u = |u|^2, \quad
u(0,x) = u_0(x)
\]
is discontinuous everywhere in $H^s(\R^d)$ if $s<-\frac{1}{4}$.
The discontinuity at zero is proved in \cite{KT} when $d=1$ (see also \cite{IU}).
From
\[
|M(\xi_1,\xi_2)|
= ||\xi_1-\xi_2|^2-|\xi_1|^2+|\xi_2|^2|
= ||\xi_1-\xi_2|^2-(\xi_1-\xi_2)\cdot (\xi_1+\xi_2)|
\]
we can apply Corollary \ref{cor:disIP} with $\varphi (r)=r^2$, $p=2$, $m=1$, $\alpha =2$, $\beta =1$.

Zheng \cite{Zhe} showed that the Cauchy problem for the quadratic nonlinear fourth order Schr\"{o}dinger equation
\begin{equation} \label{4NLS}
i \partial _tu + \Delta ^2 u = |u|^2, \quad
u(0,x) = u_0(x)
\end{equation}
is well-posed in $H^s(\R)$ with $s>-\frac{3}{4}$.
We get that this is almost optimal.
Indeed, from
\begin{align*}
|M(\xi_1,\xi_2)|
& = ||\xi_1-\xi_2|^4-|\xi_1|^4+|\xi_2|^4| \\
& = |4(\xi_1-\xi_2)\cdot \xi_2 (\xi_1\cdot \xi_2)+2|\xi_1-\xi_2|^2 (2\xi_1-\xi_2)\cdot \xi_2|,
\end{align*}
we can apply Corollary \ref{cor:disIP} with $\varphi (r)=r^4$, $p=2$, $m=1$, $\alpha =4$, $\beta =1$, which yields that ill-posedness of \eqref{4NLS} in $H^s(\R^d)$ with $s<-\frac{3}{4}$.

\section{Proof of Theorem \ref{IP}}

First, we note that \eqref{rSch} is equivalent to the following integral equation.
\begin{align*}
v=e^{-it\Delta} v_0 & + \frac{i}{2} \int_0^t e^{-i(t-t')\Delta} (v(t')-\overline{v}(t')) dt' \\
& \quad - \frac{i}{2^p} \int _0^t e^{-i(t-t')\Delta} \omega (\sqrt{-\Delta}) (v(t')+ \overline{v}(t'))^p dt' .
\end{align*}
Hence, setting
\begin{align*}
& \mathcal{L}(v_0) = e^{-it\Delta}v_0, \ \quad \mathcal{M}(v) = \frac{i}{2} \int_0^t e^{-i(t-t')\Delta} (v(t')-\overline{v}(t')) dt' , \\
& \mathcal{N}_p (v_1,\dots , v_p) = -\frac{i}{2^p} \int _0^t e^{-i(t-t')\Delta} \omega (\sqrt{-\Delta}) \prod _{j=1}^p (v_j(t')+ \overline{v}_j(t')) dt'
\end{align*}
we have
\[
v = \mathcal{L}(v_0) +\mathcal{M}(v) + \mathcal{N}_p(v, \dots , v).
\]
Let $D := L^2(\R^d) \cap \mathcal{F}^{-1} L^1(\R^d)$ and $S_T := L^{\infty}_T D$.
Since $D$ is a Banach algebra, we have
\[
\| \mathcal{L}(v_0) \| _{S_T} = \| v_0 \| _{D}, \quad
\| \mathcal{M}(v) \| _{S_T} \le T \| v \| _{S_T}, \quad
\| \mathcal{N}_p ( v, \dots, v) \| _{S_T} \le T \| v \| _{S_T}^p .
\]
Hence, the condition (A) holds with $C_{\ast}=1$, $\delta =1$, and $B= H^s (\R^d)$ ($s<0$).

From Proposition \ref{prop:gIP}, we can ignore $M$.
Let
\[
v_0 = (\log N)^{-\frac{1}{16}} N^{-s} A^{-\frac{d}{2}} \mathcal{F}^{-1} [ \bm{1}_{\widetilde{Q}_A(N)} + \bm{1}_{\widetilde{Q}_A(2N)} ]
\]
for $1 \le A \le N$, which is the same initial data in the previous section.
Note that
\begin{align*}
\mathcal{I}_p [v_0]
& = -\frac{i}{2^p} \int_0^t e^{-i(t-t')\Delta} \omega (\sqrt{-\Delta}) \left( e^{-it'\Delta} v_0 + e^{it'\Delta} \overline{v_0} \right) ^p dt' \\
& = -i \int_0^t e^{-i(t-t')\Delta} \omega (\sqrt{-\Delta}) \left( \cos (t'\Delta) v_0 \right) ^p dt'
\end{align*}
because $v_0$ is real valued.
A direct calculation yields
\begin{align*}
| \mathcal{F}[\mathcal{I}_p [v_0]](t,\xi) |
& = \frac{|\xi|^2}{1+|\xi|^2} \left| \int_{\ast} \int_0^t e^{i(t-t')|\xi|^2} \prod_{j=1}^p \cos (t'|\xi_j|^2) \widehat{v}_0(\xi_j) dt' \right| \\
& \gtrsim t \bm{1}_{Q_A(0) \backslash Q_1(0)} \int_{\ast} \prod_{j=1}^p \widehat{v}_0(\xi_j) 
\end{align*}
for $0<t \ll N^{-2}$ and $A \in [1,N]$.
Hence, we can apply the same argument as in Proposition \ref{prop:disIP} with $\varphi (r)=r^2$ and $\alpha =2$.
Accordingly, we obtain \eqref{gB} is ill-posed in $H^s(\R^d)$ if ($d=1$, $p=3$, and $s \le -\frac{1}{2}$) or ($d\in \mathbb{N}$ and $s<\min (s_c,0)$).

When $d=1,2$, $p=2$, and $s_c \le s<-\frac{1}{2}$, we set
\[
u_0 = (\log N)^{-\frac{1}{16}} N^{-s} A^{-\frac{1}{2}} \mathcal{F}^{-1} [ \bm{1}_{\widetilde{Q}_A^{(1)}(N)}]
\]
for $A\le 1 \le N$.
Since
\[
||\xi_1|^2-|\xi_2|^2| = |(\xi_1+\xi_2) \cdot (\xi_1-\xi_2)|  \lesssim \lr{NA}
\]
for $\xi_1, \xi_2 \in \widetilde{Q}_A(N)$ with $\xi_1+\xi_2 \in Q_A(0)$, we have
\begin{align*}
& |\mathcal{F}[\mathcal{I}_2[v_0]](t,\xi)| \\
& = \omega (\xi) \left| \int_{\xi_1+\xi_2=\xi} \int_0^t e^{i(t-t')|\xi|^2} \prod_{j=1}^2 \cos (t'|\xi_j|^2) \widehat{v}_0(\xi_j)  dt'  \right| \\
& \ge \frac{\omega (\xi)}{2} \bm{1}_{Q_A^{(1)}(0)} (\xi) \bigg| \int_{\xi_1+\xi_2=\xi} \int_0^t e^{i(t-t')|\xi|^2} (\cos (t' (|\xi_1|^2+|\xi_2|^2)) \\
& \hspace*{150pt} + \cos (t' (\xi_1^2-\xi_2^2))) \widehat{v}_0(\xi_1) \widehat{v}_0(\xi_2) dt' \bigg| \\
& \ge \frac{\omega (\xi)}{2} \bm{1}_{Q_A^{(1)}(0)} (\xi) \int_{\xi_1+\xi_2=\xi} \bigg\{ \int _0^t \cos ((t-t')|\xi|^2) \cos (t'(|\xi_1|^2-|\xi_2|^2)) dt' \\
& \hspace*{100pt} -  \left| \int_0^t \cos ((t-t')|\xi|^2) \cos (t' (|\xi_1|^2+|\xi_2|^2)) dt' \right| \bigg\} \widehat{v}_0(\xi_1) \widehat{v}_0(\xi_2) \\
& \gtrsim \left( t - \frac{1}{N^2} \right) \omega (\xi) \bm{1}_{Q_A^{(1)}(0)} (\xi) \mathcal{F}[v_0^2](\xi) 
\end{align*}
for $N^{-2} \lesssim t \ll \lr{NA}^{-1}$.
Hence,
\[
\| \mathcal{I}_2 [v_0] (t) \| _{H^s}
\gtrsim t (\log N)^{-\frac{1}{8}} N^{-2s} \| \xi^2 \|_{L^2(Q_A^{(1)}(0))}
\gtrsim t (\log N)^{-\frac{1}{8}} N^{-2s} A^{\frac{5}{2}}.
\]
As in Corollary \ref{cor:disIP}, it suffices to choose $t=t(N)>0$ and $A=A(N) \le 1$ satisfying
\begin{align*}
& t N^{-s} A^{\frac{1}{2}} \le (\log N)^{-\frac{1}{8}} , \quad
2 N^{-2} \le t \le \lr{NA}^{-1} (\log N)^{-\frac{1}{8}} , \\
& t (\log N)^{-\frac{1}{8}} N^{-2s} A^{\frac{5}{2}} > (\log N)^{\frac{1}{8}} .
\end{align*}
For $s_c \le s<-\frac{1}{2}$, taking $\max \{ \frac{s}{2}, \frac{2}{3} (2s+1) \} < \theta <0$, we set
\[
t = (\log N)^{-\frac{1}{8}} \min ( N^s A^{-\frac{1}{2}}, N^{-1}A^{-1} ), \quad
A=N^{\theta} .
\]
Then,
\[
t (\log N)^{-\frac{1}{8}} N^{-2s} A^{\frac{5}{2}}
= (\log N)^{-\frac{1}{4}} \min (N^{-s+2\theta}, N^{-2s-1+\frac{3\theta}{2}}) \gg (\log N)^{\frac{1}{8}},
\]
which concludes the proof.

\begin{rmk} \label{rmk:complex}
If $u$ is a complex valued function, setting
\[
v = u - i (1-\Delta)^{-1} \partial _t u , \quad
\overline{w} = u+i (1-\Delta)^{-1} \partial _t u,
\]
we rewrite \eqref{gB} as
\begin{align*}
& (i\partial _t - \Delta) v = - \frac{v-\overline{w}}{2} + \frac{1}{2^p} \omega (\sqrt{-\Delta}) (v+\overline{w})^p, \\
& (i\partial _t - \Delta) w = \frac{\overline{v}-w}{2} + \frac{1}{2^p} \omega (\sqrt{-\Delta}) (\overline{v}+w)^p .
\end{align*}
From Proposition \ref{prop:gIP}, we may ignore the parts $- \frac{v-\overline{w}}{2}$ and $\frac{\overline{v}-w}{2}$.
Owing to well-posedness in $ L^2(\R^d) \cap \mathcal{F}L^1(\R^d)$, we have
\[
v = \sum _{n=1}^{\infty} \mathcal{J}_n^1[v_0,w_0] , \quad
w = \sum _{n=1}^{\infty} \mathcal{J}_n^2[v_0,w_0] .
\]
Putting $v(0,x)=w(0,x)=v_0(x)$ as above, we have 
\[
\mathcal{J}_p^1[v_0,v_0] = \mathcal{J}_p^2[v_0,v_0] = \mathcal{I}_p[v_0].
\]
Hence, the same argument as above is applicable.
\end{rmk}

\section{Norm inflation with general initial data for the Kawahara equation}
\label{IP:Kawahara}

First, we show norm inflation at zero.
We set
\[
u_0 = (\log N)^{-1} N^{-s} \mathcal{F}^{-1} [ \bm{1}_{[N-1,N+1] \cup [-N-1,-N+1]} ] .
\]
We define $\mathcal{I}_n$ as in \eqref{iteration}.
Namely,
\begin{align*}
& \mathcal{I}_1 [u_0] (t) := e^{t(\partial_x^5-\mathfrak{b} \partial_x^3)} u_0, \\
& \mathcal{I}_n[u_0] (t) := - \sum _{\substack{n_1, n_2 \in \mathbb{N} \\ n_1+n_2=n}} \int_0^t e^{(t-t')(\partial_x^5-\mathfrak{b} \partial_x^3)} \partial _x (\mathcal{I}_{n_1}[u_0] \mathcal{I}_{n_2}[u_0]) (t') dt'
\end{align*}
for $n \ge 2$.

\begin{lem} \label{lem:boundKaw}
There exists $C_1>1$ such that
\begin{align*}
& \| \mathcal{I}_n [u_0] (t) \| _{\mathcal{F}L^1} \le (C_1 N t \| u_0 \| _{\mathcal{F}L^1} )^{n-1} \| u_0 \|_{\mathcal{F}L^1} , \\
& \| \mathcal{I}_n [u_0] (t) \| _{L^2} \le (C_1 N t \| u_0 \|_{\mathcal{F}L^1} ) ^{n-1} \| u_0 \|_{L^2}
\end{align*}
for $t>0$.
\end{lem}

\begin{proof}
Let $\{ a_n \}$ be the sequence defined by
\[
a_1 = 1, \quad
a_n = \frac{2n}{n-1} \sum _{\substack{n_1, n_2 \in \mathbb{N} \\ n_1+n_2=n}} a_{n_1} a_{n_2}
\]
for $n \ge 2$.
Owing to Lemma \ref{lem:bound-seq}, it suffices to show that
\begin{align}
& \| \mathcal{I}_n [u_0] (t) \| _{\mathcal{F}L^1} \le a_n ( N t \| u_0 \| _{\mathcal{F}L^1} ) ^{n-1} \| u_0 \|_{\mathcal{F}L^1} , \notag \\
& \| \mathcal{I}_n [u_0] (t) \| _{L^2} \le a_n (N t \| u_0 \| _{\mathcal{F}L^1} ) ^{n-1} \| u_0 \|_{L^2} . \label{bound:L2k}
\end{align}
Since $n=1$ is trivial, we assume that these estimates hold up to $n-1$.
Then, from the induction hypothesis and $\supp \mathcal{F}[\mathcal{I}_n [u_0] (t)] \subset [-2nN,2nN]$, we have
\begin{align*}
\| \mathcal{I}_n [u_0] (t) \| _{\mathcal{F}L^1}
& \le 2nN \sum _{\substack{n_1, n_2 \in \mathbb{N} \\ n_1+n_2=n}} \int_0^t \bigg\| \mathcal{I}_{n_1}[u_0] (t') \mathcal{I}_{n_2}[u_0] (t') \bigg\| _{\mathcal{F}L^1} dt' \\
& \le 2nN \sum _{\substack{n_1, n_2 \in \mathbb{N} \\ n_1+n_2=n}} \int_0^t \prod_{j=1}^2 \bigg\| \mathcal{I}_{n_j}[u_0] (t') \bigg\| _{\mathcal{F}L^1} dt' \\
& \le 2nN \sum _{\substack{n_1,n_2 \in \mathbb{N} \\ n_1+n_2=n}} \int_0^t \prod _{j=1}^2 \bigg\{ a_j ( N t' \| u_0 \| _{\mathcal{F}L^1} ) ^{n_j-1} \| u_0 \|_{\mathcal{F}L^1} \bigg\} dt' \\
& = a_n ( N t \| u_0 \|_{\mathcal{F}L^1} ) ^{n-1} \| u_0 \|_{\mathcal{F}L^1} .
\end{align*}
The estimate \eqref{bound:L2} is obtained in the same manner.
\end{proof}

The same argument as in Lemma \ref{lem:boundHs} yields
\begin{equation} \label{boundHsk}
\| \mathcal{I}_n[u_0] (t) \| _{H^s} \le C^n (Nt)^{n-1} \| u_0 \|_{\mathcal{F}L^1}^{n-2} \| u_0 \|_{L^2}^2
\end{equation}
for $n \ge 2$ and $s<-\frac{1}{2}$.
Hence, the power series $\sum _{n=1}^{\infty} \mathcal{I}_n [u_0]$ converges to $u$ in $C([0,T]; H^s(\R))$ provided that
\begin{equation} \label{cond:Ka1}
T N \| u_0 \|_{\mathcal{F}L^1} \sim T (\log N)^{-1} N^{-s+1} \ll 1.
\end{equation}
Moreover, the limit $u$ solves \eqref{Kawahara}.
In contrast, as in \cite{Kat11},
\begin{align*}
& | \mathcal{F}[\mathcal{I}_2[u_0]] (t,\xi) | \\
& = \bigg| \xi \int _0^t e^{i(t-t')(\xi^5+\mathfrak{b} \xi^3)} \int_{\xi_1+\xi_2=\xi} e^{it'(\xi_1^5+\mathfrak{b} \xi_1^3)} \widehat{u}_0(\xi_1) e^{it'(\xi_2^5+\mathfrak{b} \xi_2^3)} \widehat{u}_0(\xi_2)dt' \bigg| \\
& = |\xi| \bigg| \int_{\xi_1+\xi_2=\xi} \widehat{u}_0(\xi_1) \widehat{u}_0(\xi_2) \int_0^t e^{-t' M(\xi_1,\xi_2)} dt' \bigg|,
\end{align*}
where
\[
M(\xi_1,\xi_2) = (\xi_1+\xi_2)^5+\mathfrak{b} (\xi_1+\xi_2)^3 - \xi_1^5-\mathfrak{b} \xi_1^3-\xi_2^5-\mathfrak{b} \xi_2^3.
\]
By $\xi_1^5+\xi_2^5= (\xi_1+\xi_2)(\xi_1^4-\xi_1^3\xi_2+\xi_1^2\xi_2^2-\xi_1\xi_2^3+\xi_2^4)$, we have
\[
|M(\xi_1,\xi_2)| \lesssim N^4
\]
for $\xi_1, \, \xi_2 \in [N-1,N+1] \cup [-N-1,-N+1]$ with $\xi_1+\xi_2 \in [-1,1]$.
Therefore, we obtain
\[
| \mathcal{F}[\mathcal{I}_2[u_0]] (t,\xi) | \gtrsim |\xi| t (\log N)^{-2} N^{-2s} \bm{1}_{[-1,1]} (\xi)
\]
provided that $t=N^{s-2}$ with $s<-2$, which satisfies \eqref{cond:Ka1}.
Accordingly,
\[
\| \mathcal{I}_2[u_0] (t) \|_{H^s}
\gtrsim t (\log N)^{-2} N^{-2s} \| |\xi| \|_{L^2([-1,1])}
\sim (\log N)^{-2} N^{-s-2} .
\]
Hence, by the triangle inequality and \eqref{boundHsk},
\begin{align*}
\| u(t) \|_{H^s}
& \ge \| \mathcal{I}_2[u_0] \|_{H^s} - \| \mathcal{I}_1[u_0] \|_{H^s} - \sum _{n=3}^{\infty} \| \mathcal{I}_n[u_0] \|_{H^s} \\
& \gtrsim (\log N)^{-2} N^{-s-2} - (\log N)^{-1} -  (Nt)^2 (\log N)^{-3} N^{-3s} \\
& \sim (\log N)^{-2} N^{-s-2} ,
\end{align*}
which proves norm inflation at zero.

\begin{rmk}
Here, since we do not relay on well-posedness in $L^2(\R) \cap \mathcal{F} L^1(\R)$, we have to confirm that $\mathcal{I}_2[u_0]$ is a leading part of the expansion.
From $\| \mathcal{I}_2 [u_0] \|_{H^s} \gtrsim t \| u_0 \|_{L^2}^2$ and 
\[
\sum _{n=3}^{\infty} \| \mathcal{I}_n [u_0] \|_{H^s} \lesssim \| \mathcal{I}_3 [u_0] \|_{H^s} \lesssim (Nt)^2 \| u_0 \|_{\mathcal{F}L^1} \| u_0 \|_{L^2}^2 = t \| u_0 \|_{L^2} \times t N^2 \| u_0 \|_{\mathcal{F}L^1},
\]
we have to choose $t$ as $t N^2 \| u_0 \|_{\mathcal{F}L^1} \sim t (\log N)^{-1} N^{-s+2}<1$.

\end{rmk}

Second, we prove Theorem \ref{IP:Kaw}.
By the same argument as in the proof of Proposition \ref{prop:gIP}, it suffices to show norm inflation at any $\varphi \in \mathcal{S}(\R)$.
Set
\[
\varphi _N := \mathcal{F}[\bm{1}_{[-N,N]} \widehat{\varphi}], \quad
v_0 := u_0+ \varphi _N.
\]
Since $\supp \mathcal{F}[\mathcal{I}_n [v_0] (t)] \subset [-2nN,2nN]$, the same calculation as above yields
\begin{equation} \label{boundHsk2}
\| \mathcal{I}_n[v_0] (t) \| _{H^s} \le C^n (Nt)^{n-1} \| v_0 \|_{\mathcal{F}L^1}^{n-2} \| v_0 \|_{L^2}^2 .
\end{equation}
and
\[
v := \sum _{n=1}^{\infty} \mathcal{I}_n[v_0] \quad \text{in } C([0,T]; H^s(\R))
\]
solves \eqref{Kawahara}.

By $\| \varphi _N \|_{L^2 \cap \mathcal{F}L^1} \le \| \varphi \|_{L^2 \cap \mathcal{F}L^1}$, we have
\[
\| v_0 \|_{L^2 \cap \mathcal{F}L^1} \le 2 \| u_0 \|_{L^2 \cap \mathcal{F}L^1}
\]
for sufficiently large $N$.
Then, the same argument as in Lemma \ref{lem:boundKaw} implies the following.
\begin{lem}
There exists $C_1>1$ such that
\begin{align*}
& \| \mathcal{I}_n[u_0] (t)  - \mathcal{I}_n[v_0] (t) \| _{\mathcal{F}L^1} \le (C_1 N t \| u_0 \| _{\mathcal{F}L^1} )^{n-1} \| \varphi \|_{L^2 \cap \mathcal{F}L^1} , \\
& \| \mathcal{I}_n[u_0] (t)  - \mathcal{I}_n[v_0] (t) \| _{L^2} \le (C_1 N t \| u_0 \|_{\mathcal{F}L^1} ) ^{n-1} \| \varphi \|_{L^2 \cap \mathcal{F}L^1}
\end{align*}
for $t>0$.
\end{lem}
Accordingly, we have
\[
\| \mathcal{I}_n[u_0] (t)  - \mathcal{I}_n[v_0] (t)  \|_{H^s}
\le C^n (Nt)^{n-1} \| u_0 \|_{\mathcal{F}L^1}^{n-2} \| \varphi \|_{L^2 \cap \mathcal{F}L^1}^2
\]
for $n \ge 2$.
Hence,
\begin{align*}
\| v(t) \| _{H^s}
& \ge \| u(t) \|_{H^s} - \sum _{n=1}^{\infty} \| \mathcal{I}_n[u_0] (t)  - \mathcal{I}_n[v_0] (t)  \|_{H^s} \\
& \gtrsim (\log N)^{-2} N^{-s-2} - \| \varphi \| _{H^s} - N^{s-1} \| \varphi \| _{L^2 \cap \mathcal{F}L^1}^2 \\
& \sim (\log N)^{-2} N^{-s-2}
\end{align*}
for sufficiently large $N$.
In contrast,
\[
\| v(0) - \varphi \| _{H^s} = \| u_0 +\varphi _N-\varphi \|_{H^s} \le \| u_0 \|_{H^s} + \| (1- \bm{1}_{[-N,N]}) \widehat{\varphi} \|_{L^2}
\]
goes to zero as $N \rightarrow \infty$, which concludes the proof.


\begin{thebibliography}{99}
\bibitem{BT}
I. Bejenaru and T. Tao,
\textit{Sharp well-posedness and ill-posedness results for a quadratic non-linear Schr\"{o}dinger equation},
J. Funct. Anal. 233 (2006), no. 1, 228-259.

\bibitem{BS}
J.L. Bona and R.L. Sachs, \textit{Global existence of smooth solutions and stability of solitary waves for a generalized Boussinesq equation},
Comm. Math. Phys. 118 (1988), no. 1, 15-29.

\bibitem{CLMW}
W. Chen, J. Li, C. Miao, and J. Wu, Jiahong,
\textit{Low regularity solutions of two fifth-order KdV type equations},
J. Anal. Math. 107 (2009), 221-238. 

\bibitem{CDT}
S. Cui, D. Deng, and S. Tao,
\textit{Global existence of solutions for the Cauchy problem of the Kawahara equation with $L^2$ initial data},
Acta Math. Sin. (Engl. Ser.) 22 (2006), no. 5, 1457-1466. 

\bibitem{CG}
W. Chen and Z. Guo,
\textit{Global well-posedness and I method for the fifth order Korteweg-de Vries equation},
J. Anal. Math. 114 (2011), 121-156. 

\bibitem{CP}
A. Choffrut and O. Pocovnicu,
\textit{Ill-posedness of the cubic nonlinear half-wave equation and other fractional NLS on the real line},
arXiv:1601.02000.

\bibitem{FLS}
F. Falk, E.W. Laedke, and K.H. Spatschek,
\textit{Stability of solitary-wave pluses in shape-memory alloy},
Phys. Rev. B 36 (1987), 3031-3041.

\bibitem{Far09}
L.G. Farah,
\textit{Local solutions in Sobolev spaces and unconditional well-posedness for the generalized Boussinesq equation},
Commun. Pure Appl. Anal. 8 (2009), no. 5, 1521-1539.

\bibitem{Far09-2}
L.G. Farah,
\textit{Local solutions in Sobolev spaces with negative indices for the "good'' Boussinesq equation},
Comm. Partial Differential Equations 34 (2009), no. 1-3, 52-73.

\bibitem{GHK}
D. Geba, A.A. Himonas, and D. Karapetyan,
\textit{Ill-posedness results for generalized Boussinesq equations},
Nonlinear Anal. 95 (2014), 404-413.

\bibitem{HMO}
H. Huh, S. Machihara, and M. Okamoto,
\textit{Well-posedness and ill-posedness of the Cauchy problem for the generalized Thirring model},
Differential Integral Equations \textbf{29}, 5-6, 2016, 401-420.

\bibitem{IO}
T. Iwabuchi and T. Ogawa,
\textit{Ill-posedness for nonlinear Schr\"{o}dinger equation with quadratic non-linearity in low dimensions},
Trans. Amer. Math. Soc. 367 (2015), no. 4, 2613-2630.

\bibitem{IU}
T. Iwabuchi and K. Uriya,
\textit{Ill-posedness for the quadratic nonlinear Schrödinger equation with nonlinearity $|u|^2$},
Commun. Pure Appl. Anal. 14 (2015), no. 4, 1395-1405.

\bibitem{Kat11}
T. Kato,
\textit{Local well-posedness for Kawahara equation},
Adv. Differential Equations 16 (2011), no. 3-4, 257-287.

\bibitem{Kaw}
T. Kawahara,
\textit{Oscillatory solitary waves in dispersive media},
J. Phys. Soc. Japan, 33 (1972), 260-264.

\bibitem{Kis13}
N. Kishimoto, 
\textit{Sharp local well-posedness for the "good'' Boussinesq equation},
J. Differential Equations 254 (2013), no. 6, 2393-2433.

\bibitem{Kis}
N. Kishimoto,
\textit{A remark on norm inflation for nonlinear Schr\"{o}dinger equations},
preprint.

\bibitem{KT}
N. Kishimoto and K. Tsugawa, Kotaro,
\textit{Local well-posedness for quadratic nonlinear Schr\"{o}dinger equations and the ``good'' Boussinesq equation},
Differential Integral Equations 23 (2010), no. 5-6, 463-493. 

\bibitem{Lin}
F. Linares,
\textit{Global existence of small solutions for a generalized Boussinesq equation},
J. Differential Equations 106 (1993), no. 2, 257-293. 

\bibitem{MacOka1}
S. Machihara and M. Okamoto,
\textit{Ill-posedness of the Cauchy problem for the Chern-Simons-Dirac system in one dimension},
J. Differential Equations 258 (2015), 1356-1394.

\bibitem{MacOka2}
S. Machihara and M. Okamoto,
\textit{Sharp well-posedness and ill-posedness for the Chern-Simons-Dirac system in one dimension},
Int. Math. Res. Not. IMRN 2016 (6), 1640-1694. 

\bibitem{MacOka3}
S. Machihara and M. Okamoto,
\textit{Remarks on ill-posedness for the Dirac-Klein-Gordon system},
Dyn. Partial Differ. Equ., vol. 13, No. 3 (2016), 179-190.

\bibitem{Oh}
T. Oh,
\textit{A remark on norm inflation with general initial data for the cubic nonlinear Schr\"{o}dinger equations in negative Sobolev spaces},
arXiv:1509.08143.

\bibitem{OW}
T. Oh and Y. Whang,
\textit{On the ill-posedness of the cubic nonlinear Schr\"{o}dinger equation on the circle},
arXiv:1508.00827.

\bibitem{WCD}
H. Wang, S.B. Cui, and G. Deng,
\textit{Global existence of solutions for the Kawahara equation in Sobolev spaces of negative indices},
Acta Math. Sin. (Engl. Ser.) 23 (2007), no. 8, 1435-1446.

\bibitem{Xia}
B. Xia,
\textit{Generic ill-posedness for wave equation of power type on 3D torus},
arXiv:1507.07179.

\bibitem{Zak}
V.E. Zakharov,
\textit{On stochastization of one-dimensional chains of nonlinear oscillators},
Sov. Phys. JETP 38 (1974), 108-110.

\bibitem{Zhe}
J. Zheng,
\textit{Well-posedness for the fourth-order Schr\"{o}dinger equations with quadratic nonlinearity}, 
Adv. Differential Equations 16 (2011), no. 5-6, 467-486. 


\end{thebibliography}
\end{document}